\documentclass[a4paper,reqno,oneside]{amsart}
\usepackage{hyperref}

\newtheorem{thm}{Theorem}
\newtheorem{lem}[thm]{Lemma}
\newtheorem{prop}[thm]{Proposition}

\theoremstyle{definition}

\newtheorem{rem}[thm]{Remark}

\newcommand{\RR}{\mathbb{R}}

\newcommand{\NN}{\mathbb{N}}
\newcommand{\ZZ}{\mathbb{Z}}

\newcommand{\cN}{\mathcal{N}}
\newcommand{\cO}{\mathcal{O}}

\newcommand{\cH}{\mathcal{H}}
\newcommand{\cD}{\mathcal{D}}
\newcommand{\one}{1}

\DeclareMathOperator{\ran}{ran}

\title[]{Spectral asymptotics for $\delta$-interactions on sharp cones}

\author{Thomas Ourmi\`eres-Bonafos}
\address{Laboratoire de Math\'ematiques d'Orsay, Univ.~Paris-Sud, CNRS, Universit\'e Paris-Saclay, 91405 Orsay, France}

\email{thomas.ourmieres-bonafos@math.u-psud.fr}
\urladdr{http://www.math.u-psud.fr/~ourmieres-bonafos/}

\author{Konstantin Pankrashkin} 

\address{Laboratoire de Math\'ematiques d'Orsay, Univ.~Paris-Sud, CNRS, Universit\'e Paris-Saclay, 91405 Orsay, France
\& Laboratoire Poems, INRIA, ENSTA ParisTech, 828, Boulevard des Mar\'echaux, 91762 Palaiseau, France}

\email{konstantin.pankrashkin@math.u-psud.fr}
\urladdr{http://www.math.u-psud.fr/~pankrashkin/}

\author{Fabio Pizzichillo}
\address{BCAM -- Basque Center for Applied Mathematics, Mazarredo, 14, 48009 Bilbao Basque Country, Spain}

\email{fpizzichillo@bcamath.org}
\urladdr{http://www.bcamath.org/en/people/fpizzichillo}
\begin{document}

\keywords{Schr\"odinger operator, $\delta$-interaction, conical surface, eigenvalue, asymptotic analysis}

\begin{abstract}
We investigate the spectrum of three-dimensional Schr\"{o}dinger operators with $\delta$-interactions of constant strength supported on circular cones. As shown in earlier works, such operators have infinitely many eigenvalues below the threshold of the essential spectrum. We focus on spectral properties for sharp cones, that is when the cone aperture goes to zero, and we describe the asymptotic behavior of the eigenvalues and of the eigenvalue counting function.
A part of the results are given in terms of numerical constants appearing as solutions of transcendental equations involving modified Bessel functions.
\end{abstract}                                                            

\maketitle

\section{Introduction and main results}
%
%
%
%

%
%
%
%
For $\theta \in \big(0,\frac{\pi}{2}\big)$ we introduce the conical surface of half-aperture $\theta$ defined by
	\[
		C_\theta:= \Big\{(x_1,x_2,x_3)\in\mathbb{R}^3 : x_3= \cot \theta \sqrt{x_1^2 + x_2^2} \,\Big\}.
	\]
In the present paper
we are interested in some spectral properties of the Schr\"odinger operator with an attractive
$\delta$-potential supported on $C_\theta$.
The respective operator $L_{\theta,\alpha}$ is the unique self-adjoint operator in $L^2(\mathbb R^3)$
associated with the closed lower semibounded sesquilinear form $\ell_{\theta,\alpha}$ given by
	\[
		\ell_{\,\theta,\alpha}(u,u) = \iiint_{\RR^3}|\nabla u|^2\, dx - \alpha\iint_{C_\theta} |u|^2 d\sigma,\quad u\in H^1(\RR^3),
	\]
where $\sigma$ is the two-dimensional Hausdorff measure on $C_\theta$ and $\alpha>0$ is a constant measuring the strentgh
of the interaction. Informally, the operator $L_{\theta,\alpha}$ acts as the Laplacian, $u\mapsto -\Delta u$, in $\RR^3\setminus C_\theta$
on the functions $u$ satisfying the boundary condition $[\partial u]+\alpha u=0$ on $\Sigma$, where $[\partial u]$
is a suitably defined jump of the normal derivative, see \cite{e08} for details. As the conical surface $C_\theta$ is invariant with respect to the dilations, the operator $L_{\theta,\alpha}$
is unitarily equivalent to $\alpha^2 L_{\theta,1}$, thus in what follows we restrict ourselves to the study
of the operator $L_{\theta}:=L_{\theta,1}$ and of the form $\ell_\theta:=\ell_{\theta,1}$.
It seems that the operator $L_{\theta}$ was first considered in \cite{BEL14}:
it was shown that its essential spectrum covers the half-axis $[-\frac{1}{4},+\infty)$
and it has infinite many eigenvalues in $(-\infty,-\frac{1}{4})$.
It was shown in \cite{LOB} that the eigenvalues are increasing
in $\theta$ and that the associated eigenfunctions are invariant under the rotations around the $x_3$-axis,
and the accumulation rate of the eigenvalues to the bottom
of the essential spectrum was described: if $\cN(L_\theta,E)$ stands for the number of eigenvalues of
$L_\theta$ in $(-\infty,E)$, then
\[
\cN\Big(L_\theta,-\dfrac{1}{4}-E\Big)= \dfrac{\cot \theta}{4\pi} |\log E| + o(\log E) \text{ for } E\to 0^+.
\]
The results were then extended to $\delta$-potentials supported by non-circular conical surfaces
in \cite{el,obp}, and we refer to \cite{bpp,dobr,et,klob,p16}
for the discussion of other types of differential operators
in conical geometries. The goal of the present paper is to describe
the behavior of the eigenvalues of $L_\theta$ for the sharp cones, i.e. for the case $\theta\to 0^+$.

In order to present the main results we need to introduce several numerical constants.
As usual, by $I_n$ and $K_n$ we denote the $n$-th order modified Bessel functions.
Let $A>0$ be the unique (as shown below) solution to
\begin{equation}\label{eq-A}
I_0(A) K_0(A)+A\Big(I_1(A)K_0(A)-I_0(A)K_1(A)\Big)=0, \quad A\simeq 1.0750,
\end{equation}
and set
\begin{equation}
\label{eq-xi0}
\begin{aligned}
a_0&:=A^2 I_0(A)^2K_0(A)^2\simeq 0.2845,\\
a_1&= a_0\sqrt{\dfrac{1\mathstrut}{2A^2}+\dfrac{1}{2a_0}-2}\simeq 0.1241,\\
	\xi_0&:=\dfrac{1}{\sqrt{2} I_0(A)K_0(A)}\equiv \dfrac{1}{\sqrt{2 a_0}}\simeq 1.4252.
\end{aligned}
\end{equation}

Denote the $n$-th eigenvalue of $L_\theta$ by $E_n(L_\theta)$, then the behavior
of the individual eigenvalues is as follows:
\begin{thm}\label{thm:mainth}
For any fixed $n\in \mathbb N$ one has
\[
E_n(L_\theta) = -a_0 + a_1(2n-1) \theta + \mathcal{O}(\theta^{\frac{3}{2}}) \text{ as }\theta\to 0^+.
\]
\end{thm}
Furthermore, the following asymptotics for the eigenvalue accumulation holds:
\begin{thm}\label{thm:mainth2}
For any $\gamma\in \big(0,\frac{3}{2}\big)$ and $C>0$ there holds
\[
\cN\Big(L_\theta,-\dfrac{1}{4}-C\theta^\gamma\Big) \sim \dfrac{\gamma}{4\pi} \dfrac{|\log \theta|}{\theta}
\text{ for } \theta\to 0^+.
\]
\end{thm}

\begin{rem}
One can also show, by a technically involved but standard Agmon-type approach,
see e.g. \cite{bhbook}, that the eigenfunctions are localized near the point $(0,0,\xi_0/\theta)$:
in the simplest version, if $u_{n,\theta}$ is an $L_2$-normalized eigenfunction of $L_\theta$ for the eigenvalue $E_n(\theta)$
with a fixed $n$, then for small $\theta$ there holds, with suitable $a,b,c>0$,
\[
\iiint_{\RR^3} \big|u_{n,\theta}(x)\big|^2 \exp\bigg( a|x_1|+a|x_2| + \dfrac{b}{\theta}\Big|x_3-\dfrac{\xi_0}{\theta}\Big|\bigg)  dx <c.
\]
\end{rem}

Our proofs are based on a rather straightforward application of the Born-Oppen\-heimer strategy, see e.g. \cite{ray} for an extensive discussion.
In Section \ref{sec:auxope} we recall some constructions related to the min-max principle
and provide a detailed study of several one- and two-dimensional operators. A part of the study is based
on involved operations with modified Bessel functions. The information obtained is then used
in Section~\ref{sec:proofth} to prove Theorems~\ref{thm:mainth} and~\ref{thm:mainth2}.
The proof scheme is quite close to the one used in \cite{DR}
for the study of the two-dimensional counterpart of the problem, i.e. for the $\delta$-interaction
supported on the boundary of a sharp infinite sector, but with essential differences due
to the properties of the associated models operators.
\section{Auxiliary constructions}\label{sec:auxope}

\subsection{Min-max principle}

Let us recall some constructions related to the min-max principle for self-adjoint operators,
see e.g. \cite[Chapter XIII.1]{RS78}.

Let $\cH$ be an infinite-dimensional Hilbert space and $E\in\RR$. If $T$ is a self-adjoint operator in $\cH$, then we denote
by $\cD(T)$ its domain and by $E_n(T)$ we denote the $n$-th eigenvalue of $T$ when enumerated in the non-decreasing
order and counted according to the multiplicities. The symbol $\cN(T,E)$ will stand for the dimension of the range of the spectral projector
of $T$ on $(-\infty,E)$. If $T$ is lower semibounded and $E<\inf\sigma_\mathrm{ess}(T)$,
then $\cN(T,E)$ is exactly the number of eigenvalues of $T$ (counting the multiplicities)
in $(-\infty,E)$, otherwise one has $\cN(T,E)=+\infty$. Remark that
$\cN(T_1\oplus T_2,E)=\cN(T_1,E)+\cN(T_2,E)$ for any two self-adjoint operators
$T_1$ and $T_2$ and any $E\in\RR$.
The function $E\mapsto \cN(T,E)$ is usually called the eigenvalue counting function for $T$.

If the operator $T$ in $\cH$ is generated by a closed lower semibounded
sesquilinear form $t$ defined on the domain $\cD(t)$, then the following variational characterization
of the eigenvalues holds (min-max principle): for $n\in\NN$ set
\[
\Lambda_n(T):=\inf_{\substack{V\subset \cD(t)\\ \dim V=n}}\sup_{\substack{u\in V\\u\ne 0}} \dfrac{t(u,u)}{\|u\|^2_{\cH}},
\]
then $E_n(T)=\Lambda_n(T)$ if $\Lambda_n(T)<\inf\sigma_\mathrm{ess}(T)$, otherwise
one has $\Lambda_m(T)=\inf\sigma_\mathrm{ess}(T)$ for all $m\ge n$.
We denote $E_n(t):=E_n(T)$, $\Lambda_n(t):=\Lambda_n(T)$,  $\cN(t,E):=\cN(T,E)$.

For two sesquilinear forms $t_1$ and $t_2$, their direct sum
$t_1\oplus t_2$ is the sesquilinear form defined on $\cD(t_1\oplus t_2):=\cD(t_1) \times \cD(t_2)$
by
\[
(t_1\oplus t_2)\big((u_1,u_2),(u_1,u_2)\big):=t_1(u_1,u_1)+t_2(u_2,u_2).
\]
If $T_1$ and $T_2$ are the operators associated with $t_1$ and $t_2$,
then the operator associated with $t_1\oplus t_2$ is $T_1\oplus T_2$.
The form inequality $t_1\ge t_2$ means that $\cD(t_1)\subseteq \cD(t_2)$
and $t_1(u)\ge t_2(u)$ for all $u\in \cD(t_1)$. By the min-max principle,
the form inequality implies the respective inequality for the eigenvalues,
$E_n(t_1)\ge E_n(t_2)$ for any ${n\in\NN}$, and the reverse inequality
for the eigenvalue counting functions,
$\cN(t_1,E)\le \cN(t_2,E)$ for all $E\in\RR$.

\subsection{One dimensional semi-classical operator}
Let us recall a classical result on the harmonic approximation of one-dimensional operators,
see \cite{bhbook}.

\begin{prop}\label{prop:harmapprox}
Let $(a,b)\subset\RR$ be a finite or infinite non-empty open interval
and $U:(a,b)\to \mathbb R$ be a bounded $C^\infty$ function having
a unique minimum at $\xi\in(a,b)$, which is non-degenerate, i.e.
	\[
		U(x) > U(\xi) \text{ for all } x\in(a,b)\setminus\{\xi\},
		\quad \quad U''(\xi) > 0,
	\]
	and such that
	\[
	\liminf_{x\to a^+} U(x)> U(\xi), \quad
	\liminf_{x\to b^-} U(x)> U(\xi).
	\]
Let $T_h$ be the self-adjoint operator in $L^2(a,b)$ given by
\[
T_h=-h^2\dfrac{d^2}{dx^2}+U
\]
with any self-adjoint $h$-independent boundary conditions at $a$ and $b$, then
for any fixed $n\in\NN$ there holds
	\[
		E_n(T_h) = U(\xi) + (2n-1)\sqrt{\frac{U''(\xi)}{2}}\, h + \mathcal{O}(h^{\frac{3}{2}}) \text{ as } h\to 0^+.
	\]
\end{prop}

\subsection{Two dimensional $\delta$-interaction on a circle of varying radius}\label{circ2}

The operator $B_{R,\beta}$ that will be of importance in what follows is the two-dimensional Schr\"{o}dinger with a $\delta$-interaction of strength $(-\beta)$ supported by a circle of radius $R>0$. It is defined via the associated sesquilinear form
	\[
		b_{R,\beta}(u,u) =\iint_{\mathbb{R}^2}|\nabla u|^2 dx_1\,dx_2 - \beta \int_{|x|=R} |u|^2 d\sigma,\quad u\in H^1(\mathbb{R}^2),
	\]
with $d\sigma$ being the arclength element. One may show that the operator $B_{R,\beta}$ is the Laplacian acting
of the functions $u$ satisfying the transmission conditions
\begin{equation}
   \label{trans}
u|_{r=R^-}=u|_{r=R^+}=:u|_{r=R}, \quad
\dfrac{\partial u}{\partial r}\Big|_{r=R^+}-\dfrac{\partial u}{\partial r}\Big|_{r=R^-}=-\beta u|_{r=R},
\end{equation}
with $r:=|x|$, see e.g. \cite{ET04}. It is easy to see that the essential spectrum of $B_{R,\beta}$ is $[0,+\infty)$.
We will need some information about the dependence of the first eigenvalues on the radius $R$.
In what follows we set
\[
\mu_n(R;\beta):=\Lambda_n(B_{R,\beta}).
\]
Recall that the constants $a_0$, $a_1$ and $\xi_0$ are defined in \eqref{eq-xi0}.

\begin{prop}\label{prop:2Dmodel} 

\textup{(a)} For any $R>0$ and $\beta>0$ one has $\mu_1(R,\beta)<0$, i.e. it is the first eigenvalue,
		which is simple, and the map $(0,+\infty)\ni R \mapsto \mu_1(R;\beta)\in\mathbb{R}$ is $C^\infty$.
		Furthermore, if $\Phi_{R,\beta}$ is the associated eigenfunction chosen normalized and non-negative,
		then the map $(0,+\infty)\ni R\mapsto \Phi_{R,\beta}\in L^2(\mathbb R^2)$ is $C^1$.

\smallskip

\noindent 
\textup{(b)} The function $(0,+\infty)\ni R\mapsto\mu_1(R;\beta)$ has a unique minimum at
		\[R=\xi:=\dfrac{\sqrt 2}{\beta} \xi_0
		\]
		with $	\mu_1(\xi;\beta)=-\beta^2 a_0$ and $\mu''_1(\xi;\beta)=2\beta^4 a_1^2$,
    and
		\[
				\lim_{R\rightarrow 0^+} \mu_1(R;\beta) = 0, \quad \mu_1(R;\beta) = -\frac{1}{4}\,\beta^2 -\frac1{4R^2} +\mathcal{O}\Big( \frac1{R^3}\Big) \text{ as } R\to +\infty.
		\]

\smallskip

\noindent
\textup{(c)} For any $a>0$ there exists $M_a>0$ such that $||\partial_R \Phi_{R,\beta}||_{L^2(\RR^2)}\le M_a$  for all $R\in(a,+\infty)$.

\smallskip

\noindent 
\textup{(d)} For any $R>0$ and $\beta>0$ there holds $\mu_2(R;\beta)> -\frac{1}{4}\beta^2$.

\end{prop}

\begin{proof}
During the proof we will omit the dependence on the parameter $\beta$.

Recall first that the operator can be studied using a separation of variables, and we reproduce briefly  a part of computations from~\cite[\S 2.1]{ET04}.
By introducing the polar coordinates $(r,\theta)$ centered at the origin one may look for eigenfunctions of the form $\phi(r,\theta)=\rho_m(r)e^{in\theta}$, $m\in \ZZ$.
Outside the ring $r=R$ the function $\phi$ must satisfy the free Schr\"odinger equation $-\Delta \phi=-k^2 \phi$, $k>0$, hence, one should look for the functions $\rho_n$
having the form
\[
\rho_m(r)=\begin{cases}
c_1 I_m(k r), & r<R,\\
c_2 K_m(kr), & r>R,
\end{cases}
\]
with $I_m$ and $K_m$ being the modified Bessel functions of the respective orders and $c_1,c_2\in\mathbb{R}$.
Taking into account the transmission conditions \eqref{trans} and the Wronskian identities
\begin{equation}
 \label{wron}
K'_{m}(x)I_m(x)-I'_m(x)K_m(x)=\dfrac{1}{x},
\end{equation}
see  \cite[Eqs. 10.28.2, 10.29.3]{Ol10}, 
one arrives at the spectral condition
\[
H_m (kR)=\dfrac{1}{\beta R}, \quad H_m=I_m K_m.
\]
Recall that one has the identities $I_{-m}=I_m$ and $K_{-m}=K_m$, hence, $H_{-m}=H_m$, and the inequalities
$H_{m-1}(x)>H_{m}(x)$ for all $m\in\NN$ and $x>0$,
see e.g. \cite[Theorem 2]{b2}.

Now let us address each point separately.

(a) The first eigenvalue $E_1(B_{R,\beta})$ is represented as $E_1(B_{R,\beta})=-k_1(R)^2$ with $k_1(R)>0$ being the solution of
\begin{equation}
    \label{impl1}
H_0(k R)=\dfrac{1}{\beta R}.
\end{equation}
By \cite[Sections 10.30 and 10.40]{Ol10} we have
\begin{equation}\label{eqn:asymptotic_H}
H_0(t)\sim -\log t \text{ for }  t\to 0,\quad
H_0(t)=\dfrac{1}{2t}\left(  1+\dfrac{1}{8t^2}+\cO\Big(\dfrac{1}{t^3}\Big) \right) \text{ for } t\to +\infty,
\end{equation}
implying $H_0(t)\to +\infty$ for $t\to 0$ and $H_0(t)\to 0$ for $t\to+\infty$. Furthermore,
by \cite[Theorem 1]{Bar09} the function $H$ is strictly completely monotonic on $(0,+\infty)$, that is for any $t\in (0,\infty)$ and $n\in \mathbb N$
one has  $(-1)^n H^{(n)}_0(t)>0$. In particular, it follows that $H_0:(0,\infty)\mapsto (0,\infty)$ is a diffeomorphism and a strictly decreasing function. Moreover, thanks to \eqref{eqn:asymptotic_H}, when $t\to 0$ we get
\begin{equation}\label{eqn:asymptotic_H^-1}
H_0^{-1}(t)=\frac{1}{2t}\left(1+\frac{t^2}{2}+\mathcal{O}(t^4)\right).
\end{equation}
Thus $k_1$ is uniquely defined by
\begin{equation}\label{eqn:def_k_1}
k_1(R)=\frac{1}{R}H^{-1}_0\left(\frac{1}{\beta R}\right)
\end{equation}
and is infinitely smooth. The respective positive normalized  eigenfunction $\Phi_{R}$ is given then by
\begin{equation}
    \label{eq-phi}
\Phi_R(x)=\alpha_R\begin{cases}
I_0\big(k_1 |x|\big)K_0(k_1R), & |x|<R,\\
I_0(k_1 R)K_0\big(k_1 |x|\big), & |x|>R,
\end{cases}
\end{equation}
where $\alpha_R>0$ is a normalization constant, and a simple computation shows that it can be differentiated in $R$.
	
(b)
The analysis is based on the implicit equation \eqref{impl1}.
Consider the function 
\begin{equation}
  \label{eqgr}
g(R):=R\ k_1(R)=H^{-1}_0\left(\frac{1}{\beta R}\right),
\end{equation}
which is strictly increasing and infinitely smooth on $(0,+\infty)$ with
	\begin{equation}\label{eqn:limitg}
		\lim_{t\rightarrow0} g(t) = 0,\quad \lim_{t\rightarrow +\infty}g(t) = +\infty.
	\end{equation}
By multiplying both sides of \eqref{impl1} by $\beta g(R)$ one arrives at
\begin{equation}
  \label{ekf}
k_1=\beta \  F\circ g, \quad F(t):= t\  H_0(t).
\end{equation}
Combining \eqref{eqn:asymptotic_H},  \eqref{eqn:limitg} and \eqref{eqn:asymptotic_H^-1} we get the expected limit for $R\to 0$
and for $R\to+\infty$
\begin{equation}\label{eqn:asymptotic_k_1}
k_1(R)=\frac{1}{2}\beta
\left(
1+\frac{1}{2\beta R^2}+\mathcal{O}\left(\frac{1}{R^4}\right)\right).
\end{equation}
which gives the asymptotic behaviour of $\mu_1(R)$ for $R\to+\infty$.

Remark that the functions $t\mapsto t H_m(t)$ were studied earlier by various authors, see e.g. \cite{bez,Har77}. In particular,
by \cite[Theorem 4.2]{Har77} there exist $t_0$ and $t_1$ with $\frac{1}{2}<t_0<t_1$ such that
\begin{gather*}
	F'(t) < 0 \ \text{for} \ t\in(0,t_0), \quad F'(t) > 0  \ \text{for} \  t\in(t_0,+\infty),\\
		F''(t) < 0 \ \text{for} \ t\in(0,t_1), \quad F''(t) > 0 \ \text{for} \ t\in(t_1,+\infty),
\end{gather*}
and it follows that $t_0$ is  the unique maximum of $F$ on $(0,+\infty)$ and is the unique solution of $F'(t)=0$, and that $F''(t_0)<0$. 
Due to identities
\begin{equation}
  \label{eqik}
I_0'=I_1, \quad K_0'=-K_1,
\end{equation}
see  \cite[Eq. 10.29.3]{Ol10}, one has $F'(t)=I_0(t)K_0(t)+t\big(I_1(t)K_0(t)-I_0(t)K_1(t)\big)$,
which implies that $t_0$ coincides with the constant $A$ in \eqref{eq-A}.

Using the expression
\begin{equation}
  \label{mu1}
\mu_1 = -k_1^2\equiv -\beta^2(F\circ g)^2
\end{equation}
one concludes that the function $(0,+\infty)\ni R\mapsto \mu_1(R)$ has a unique minimum at $R=\xi$,
where $\xi$ is chosen by the condition $g(\xi)\equiv \xi k_1(\xi)=A$. Using the identity \eqref{impl1}
for $R=\xi$ one has then
\[
H_0(A)=H_0\big(\xi k_1(\xi)\big)=\dfrac{1}{\beta \xi}, \quad \text{i.e. } \xi=\dfrac{1}{\beta H_0(A)}\equiv \dfrac{1}{\beta I_0(A)K_0(A)}= \dfrac{\sqrt 2}{\beta} \xi_0,
\]
while
\[
\min_{R>0} \mu_1(R)=-\big(\max_{R>0} k_1(R) \big)^2=-\beta^2\big(\max_{t>0} F(t) \big)^2
=-\beta^2 F(A)^2=-\beta^2 a_0.
\]
Furthermore, with the help of \eqref{mu1} one has
\begin{align*}
\mu'_1(R)&=-2\beta^2 F\big(g(R)\big) F'\big(g(R)\big) g'(R),\\
\mu''_1(R)&=-2\beta^2 \Big( F'\big(g(R)\big) g'(R) \Big)^2\\
&\qquad -2\beta^2 F\big(g(R)\big) \bigg( F''\big(g(R)\big) g'(R)^2+ F'\big(g(R)\big) g''(R)\bigg),
\end{align*}
and using $g(\xi)=A$ and $F'(A)=0$ we arrive at
$\mu''_1(\xi)= -2\beta^2 F(A) F''(A) g'(\xi)^2$.
It follows from \eqref{eqgr} and \eqref{ekf} that
\[
g'(R)=k_1(R)+Rk'_1(R),\quad
g'(\xi)=k_1(\xi)+\xi\cdot 0=\beta F(g(\xi))=\beta F(A),
\]
resulting in
\begin{equation}
        \label{mu12}
\mu''_1(\xi)= -2\beta^4 F(A)^3 F''(A).
\end{equation}
We have $F''(t)=\big(t H_0(t)\big)''=\big(H_0(t)+t H'_0(t)\big)'=2H'_0(t)+t H''_0(t)$.
Furthermore, using the definition of $A$ we have $H'_0(A)=-H_0(A)/A$. With the
help of the identities \eqref{eqik} and
\[
I_1'(t)=I_0(t)-\dfrac{I_1(t)}{t}, \quad -K'_1(t)=K_0(t)+\dfrac{K_1(t)}{t},
\]
see \cite[Eq. 10.29.2]{Ol10}, we obtain
\begin{multline*}
H''_0(A)=(I_0 K_0)''(A)=(I_1 K_0-I_0 K_1)'(A)= (I_1'K_0 -I_0 K'_1 -2I_1 K_1)(A)\\
=2I_0(A)K_0(A)-2I_1(A)K_1(A)-\dfrac{I_1(A)K_0(A)-I_0(A) K_1(A)}{A}.
\end{multline*}
Using the definition of $A$ one arrives at
\begin{multline*}
H''_0(A)= 2I_0(A)K_0(A)-2I_1(A)K_1(A)+\dfrac{I_0(A)K_0(A)}{A^2}\\
=I_0(A)K_0(A) \Big(2+\dfrac{1}{A^2}-2\dfrac{I_1(A)K_1(A)}{I_0(A)K_0(A)}\Big).
\end{multline*}
To simplify further we rewrite the condition \eqref{eq-A} for $A$ as
\[
\dfrac{K_1(A)}{K_0(A)}-\dfrac{I_1(A)}{I_0(A)}=\dfrac{1}{A},
\]
then by using the Wronskian identity \eqref{wron} combined with \eqref{eqik} one arrives at
\begin{align*}
\dfrac{I_1(A)K_1(A)}{I_0(A)K_0(A)}&=\dfrac{1}{4} \Bigg(
\bigg(\dfrac{K_1(A)}{K_0(A)}+\dfrac{I_1(A)}{I_0(A)}\bigg)^2 - \bigg(\dfrac{K_1(A)}{K_0(A)}-\dfrac{I_1(A)}{I_0(A)}\bigg)^2
\Bigg)\\
&=\dfrac{1}{4} \Bigg(
\bigg(\dfrac{K_1(A) I_0(A)+I_1(A)K_0(A)}{I_0(A)K_0(A)}\bigg)^2 - \dfrac{1}{A^2}\Bigg)\\
&=\dfrac{1}{4} \Bigg(
\bigg(\dfrac{1}{AI_0(A)K_0(A)}\bigg)^2 - \dfrac{1}{A^2}\Bigg)
=\dfrac{1}{4} \Big(\dfrac{1}{a_0} - \dfrac{1}{A^2}\Big),
\end{align*}
which then gives
\[
H''_0(A)=I_0(A)K_0(A)\Big(2+\dfrac{3}{2A^2}-\dfrac{1}{2a_0}\Big).
\]
and
\begin{multline*}
F''(A)=2H'_0(A)+A H''_0(A)=-\dfrac{2H_0(A)}{A}+A H''_0(A)\\
=-\dfrac{2 I_0(A)K_0(A)}{A}+A I_0(A)K_0(A)\Big(2+\dfrac{3}{2A^2}-\dfrac{1}{2a_0}\Big)\\
=A I_0(A)K_0(A)\Big(2-\dfrac{1}{2A^2}-\dfrac{1}{2a_0}\Big).
\end{multline*}
Hence, by \eqref{mu12},
\[
\mu''_1(\xi)= 2\beta^4 A^4 I_0(A)^4 K_0(A)^4\Big(\dfrac{1}{2A^2}+\dfrac{1}{2a_0}-2\Big)
=2\beta^4 a_1^2.
\]

(c) Thanks to (a) it is sufficient to show that $\limsup_{R\to+\infty}||\partial_R \Phi_{R}||_{L^2}<\infty$.

We start by an auxiliary estimate: for any $n,m\in \mathbb{N}$ and $\alpha>0$ there holds, for $t\to+\infty$,
	\begin{equation}\label{eqn:asymptotic_I_nK_m}
		K_n(t)^2\displaystyle\int_0^t I_m(s)^2 s^\alpha \ ds\sim\frac{t^{\alpha-2}}{8},
		\quad
		I_m(t)^2\displaystyle\int_t^{+\infty} K_n(s)^2 s^\alpha \ ds\sim\frac{t^{\alpha-2}}{8}.
	\end{equation}
In fact, from \cite[Eqs. 10.40.1 and 10.40.2]{Ol10} we know that when $t\rightarrow+\infty$
	\begin{gather}\label{eqn:asymptotic_bessel}
		\begin{aligned}
			I_m(t)&=
			\frac{e^t}{(2\pi t)^{\frac{1}{2}}}
			\left(
			1
			-\frac{a_1(m)}{t}
			+\frac{a_2(m)}{t^2}
			-\frac{a_3(m)}{t^3}	
			+\frac{a_4(m)}{t^4}		
			+\cO\Big(\frac{1}{t^5}\Big)\right),
			\\
			K_m(t)&= e^{-t}\left(\frac{\pi}{2 t}\right)^{\frac{1}{2}}
			\left(
			1
			+\frac{a_1(m)}{t}
			+\frac{a_2(m)}{t^2}
			+\frac{a_3(m)}{t^3}		
			+\frac{a_4(m)}{t^4}		
			+\cO\Big(\frac{1}{t^5}\Big)\right),\\
		\end{aligned}
		\end{gather}		
Where for all $k\in \mathbb{N}$ we set 
\[
a_k(m):=\frac{\prod_{j=1}^k \left(4m^2-(2j-1)^2\right)}{(k!)8^k}	.
\]
Moreover, from \cite[Eq. 10.29.3]{Ol10}, we have 
\begin{equation}\label{eqn:derivative_K_n}
		I'_m(t)=I_{m+1}(t)+\frac{m}{t}I_{m}(t), \quad K'_n(t)=-K_{n+1}(t)+\frac{n}{t}K_{n}(t).
\end{equation}			
By combining L'H\^opital's rule with \eqref{eqn:asymptotic_bessel} and \eqref{eqn:derivative_K_n} we arrive at
	\[
		K_n(t)^2\int_0^t I_m(s)^2 s^\alpha  \ ds \sim
			\dfrac{\displaystyle\int_0^t I_m(s)^2 s^\alpha  \ ds}{\dfrac{\mathstrut1}{K_n(t)^2}}\sim
			\dfrac{\Big(K_n(t) I_m(t)\Big)^2 t^\alpha}{-2\dfrac{ K'_n(t)}{K_n(t)}}\sim \frac{t^{\alpha-2}}{8}.
	\]
The second estimate in~\eqref{eqn:asymptotic_I_nK_m} is proved in an analogous way.

We change slightly the representation \eqref{eq-phi} for $\Phi_R$. Namely, denote
	\begin{equation}
	  \label{eq-fr}
		f_R(r):=
			\begin{cases}
				\sqrt{R}\ K_0\big(k_1 R\big)I_0\big(k_1 r\big)&\text{for}\ r\leq R,\\
				\sqrt{R}\ I_0\big(k_1 R\big)K_0\big(k_1r\big)&\text{for}\ r\geq R,
			\end{cases}
	\end{equation}
with $k_1(R)$ defined in \eqref{eqn:def_k_1} and
\[
c_R:= \int_0^\infty f_R(r)^2 r\ dr\equiv  \dfrac{1}{2\pi}\iint_{\RR^2} f_R\big(|x|\big)^2\, dx=\dfrac{1}{2\pi}\big\|f_R \big(|x|\big)\big\|^2_{L^2(\RR^2)},
\]
then
\[
\Phi_{R}(x)=\dfrac{1}{\sqrt{\mathstrut 2\pi c_R}}\, f_R\big(|x|\big).
\]
Assume that there exist $b_1,b_2\in(0,+\infty)$ such that for large $R$ there holds
	\begin{align}\label{eqn:estimate_constant}
			b_1<c_R<b_2,\\
			\label{eqn:estimate_derivative_constant}
			||\partial_R f_R(|x|)||_{L^2(\mathbb{R}^2)}<b_2,
	\end{align}
then using the triangle inequality one can estimate
	\begin{align*}
		\big\|\partial_R \Phi_R\big\|_{L^2(\RR^2)}
			&\leq \dfrac{1}{\sqrt{\mathstrut 2\pi c_R}} \Big\|\partial_R f_R \big(|x|\big)\Big\|_{L^2(\RR^2)} 
					+\dfrac{|\partial_R c_R|}{2\ \sqrt{\mathstrut 2\pi c_R}}\big\|f_R \big(|x|\big)\big\|_{L^2(\RR^2)}\\
			&=
			    \dfrac{1}{\sqrt{\mathstrut 2\pi c_R}}  \Big\|\partial_R f_R \big(|x|\big)\Big\|_{L^2(\RR^2)} + \dfrac{1}{2}\,|\partial_R c_R|.
	\end{align*}
	Furthermore, 
	\begin{multline*}
	 |\partial_R c_R|= \bigg| \partial_R \Big\langle f_R \big(|x|\big), f_R \big(|x|\big)\Big\rangle_{L^2(\RR^2)}\bigg|
	= 2\bigg|\Big\langle \partial_R f_R \big(|x|\big), f_R \big(|x|\big)\Big\rangle_{L^2(\RR^2)}\bigg|\\
	\le 2\Big\|\partial_R f_R \big(|x|\big)\Big\|_{L^2(\RR^2)}\cdot\Big\|f_R \big(|x|\big)\Big\|_{L^2(\RR^2)}\\
	= 2\sqrt{2\pi c_R} \Big\|\partial_R f_R \big(|x|\big)\Big\|_{L^2(\RR^2)}\le 2b_2 \sqrt{2\pi b_2},
	\end{multline*}
	implying 
	\[
	\big\|\partial_R \Phi_{R}\big\|_{L^2(\RR^2)}\le \dfrac{b_2}{\sqrt{2\pi b_1}} + b_2\sqrt{2\pi b_2}.
	\]
Therefore, it is sufficient to prove the inequalities \eqref{eqn:estimate_constant} and  \eqref{eqn:estimate_derivative_constant}.

Thanks to \eqref{eqn:asymptotic_I_nK_m}, the function
	\[
		G(t):=K_0(t)^2\int_0^t I_0(s)^2s\ ds+I_0(t)^2\int_t^\infty K_0(s)^2s\ ds
	\]
 satisfies
\begin{equation}\label{eqn:asymptotic_G}
G(t)\sim\frac{1}{4t} \text{ for } t\to +\infty.
\end{equation}
Using the change of variable $s=k_1(R) r$ we arrive at
	\begin{equation}\label{eqn:integration_c_R}
			c_R=\dfrac{R}{k_1(R)^2}\, G\big(k_1(R)R\big).
	\end{equation}
Using $\lim_{R\rightarrow+\infty}k_1(R)=\frac{1}{2}\,\beta$ and combining \eqref{eqn:integration_c_R} with
\eqref{eqn:asymptotic_G} we conclude that $\lim_{R\rightarrow+\infty}c_R = 2\beta^{-3}$, which proves \eqref{eqn:estimate_constant}.

We now prove the remaining inequality \eqref{eqn:estimate_derivative_constant}. Let us study first the asymptotic behavior
of $k_1'(R)$ for $R\to +\infty$. Using \eqref{eqn:def_k_1} one has
\[
k_1'(R)=-\frac{k_1(R)}{R}-\dfrac{1}{\beta R^3 H'_0\big(k_1(R)R\big)}, \quad
H_0(t):=I_0(t)K_0(t).
\]
Since $H'_0(t)=I_1(t)K_0(t) - I_0(t)K_1(t)$, thanks to \eqref{eqn:asymptotic_bessel} we can affirm that 
\[
H'_0(t)=-\frac{1}{2t^2}-\frac{3}{16t^4} + \cO\left(\frac{1}{t^6}\right) \quad t\to +\infty,
\]
which, combined with \eqref{eqn:asymptotic_k_1}, implies that
\begin{equation}\label{eqn:asymptotic_k_1'}
	k'_1(R)R=\cO\left(\dfrac{1}{R^2}\right), \quad R\to +\infty.
\end{equation}
The direct derivation of \eqref{eq-fr} with the help of \eqref{eqik} gives
\begin{equation}
     \label{eq-drf}
\partial_R f_R\big(|x|\big)=\dfrac{1}{2R} f_R \big(|x|\big) + \sqrt{R} k_1' F_R\big(|x|\big) + \sqrt{R}\big(k'_1 R+k_1\big) G_R\big(|x|\big)
\end{equation}
with
\[
F_R(r)=\begin{cases}
r I_1(k_1 r)K_0(k_1R), & r<R,\\
-r I_0(k_1 R)K_1(k_1 r), & r>R,
\end{cases}
\quad
G_R(r)=\begin{cases}
-I_0(k_1 r)K_1(k_1R), & r<R,\\
 I_1(k_1 R)K_0(k_1 r), & r>R.
\end{cases}
\]
Set, for $t>0$,
\begin{align*}
	\widetilde F(t)&:=
		K_0(t)^2\int_0^t I_1(s)^2 s^3 \ ds+I_0(t)^2\int_t^{+\infty} K_1(s)^2 s^3 \ ds,\\
	\widetilde G(t)&:=
		K_0(t)^2\int_0^t I_1(s)^2 s \ ds+I_0(t)^2\int_t^{+\infty} K_1(s)^2 s \ ds,
\end{align*}
then
\[
\Big\|F_R\big(|x|\big)\Big\|^2_{L^2(\RR^2)}=\dfrac{2\pi}{k_1^4}\,\widetilde F(k_1 R),
\quad
\Big\|G_R\big(|x|\big)\Big\|^2_{L^2(\RR^2)}=\dfrac{2\pi}{k_1^2}\,\widetilde G(k_1 R).
\]
Thanks to \eqref{eqn:asymptotic_I_nK_m} we get $\widetilde{F}(t)\sim\frac{1}{4}\, t$
and $\widetilde G(t) \sim\frac{1}{4}\,t^{-1}$ for $t\to+\infty$, hence,
\[
\sqrt{R}\Big\|F_R\big(|x|\big)\Big\|_{L^2(\RR^2)}= \cO(R),
\quad
\sqrt{R}\Big\|G_R\big(|x|\big)\Big\|_{L^2(\RR^2)}=\cO(1),
\quad R\to +\infty.
\]
Using the triangle inequality in \eqref{eq-drf} one has
\begin{align*}
\Big\|\partial_R f_R\big(|x|\big)\Big\|_{L^2(\RR^2)}&
\le \dfrac{\sqrt{2\pi c_R}}{2R} +| k_1'| \sqrt{R}  \Big\|F_R\big(|x|\big)\Big\|_{L^2(\RR^2)}\\
&\quad + |k'_1 R+k_1| \sqrt{R} \Big\|G_R\big(|x|\big)\Big\|_{L^2(\RR^2)}\\
&=\cO(R^{-1})+\cO(R^{-3})\cO(R) + \cO(1)\cO(1)=\cO(1).
\end{align*}

	(d) As follows from the separation of variables given at the beginning of the proof, the second eigenvalue of $B_{R}$ is written as $E_2(B_R)=-k_2(R)^2$
	with $k_2(R)>0$ determined by the condition
		\begin{equation}
		   \label{h1k2}
		 H_1(k_2 R)=\dfrac{1}{\beta R}.
		\end{equation}
		By \cite[Corollary 2.2]{PM50} the function $H_1$ is strictly decreasing, and by \cite[Sections 10.30 and 10.40]{Ol10}
one has $\lim_{t\rightarrow0} H_1(t) = \frac12$ and $\lim_{t\to+\infty}H_1(t)=0$,
i.e. $H_1:(0,+\infty) \rightarrow (0,\frac 12)$ is invertible. It follows that the solution $k_2$ (and, hence, the second eigenvalue of $B_R$)
exists if and only if $\beta R>2$. Furthermore, by multiplying both sides of \eqref{h1k2} by $\beta R k_2(R)$
one arrives as $k_2(R)=\beta F\big(Rk_2(R)\big)$ with $F(t)=t H_1(t)$, and by \cite[Theorem 4.1]{Har77} one has $0<F(t)<\frac{1}{2}$ for all $t>0$.
\end{proof}

\subsection{Weyl-type asymptotics with a moving threshold}

We are going to prove the following result inspired by the constructions of~\cite{KS88}:

\begin{prop}\label{prop:weyl}
Let $a>0$  and $v\in L^\infty(a,+\infty)$ be real-valued with
\begin{equation}
   \label{eq-vdelta}
\big|v(x)\big|=o\bigg(\dfrac{1}{x^2}\bigg) \text{ as } x\to+\infty
\end{equation}
Denote by $T_h$ the operator in  $L^2(a,+\infty)$
acting by
\[
(T_hu)(x)= -h^2u''(x)+ \Big(-\dfrac{1}{4x^2}+v(x)\Big) u(x)
\]
with any self-adjoint boundary condition at $a$,
then for any $\gamma\in(0,2)$ and $C>0$ there holds
	\[
		\cN \Big(T_h,-Ch^\gamma\Big)
		\sim \dfrac{\gamma}{4\pi}\dfrac{|\log h|}{h} 
		\text{ for } h\to 0^+.
	\]
\end{prop}

During the proof we adopt the following notation: for a non-empty open interval $\Omega\subset\RR$ (bounded or unbounded),
a potential $V:\Omega\to \RR$ and parameters $h>0$ and $E\in\RR$ we denote by
\[
\cN\Big(h^2D^2+V, \Omega, E\Big)
\]
the number, counting the multiplicities, of the eigenvalues in $(-\infty,E)$ of the self-adjoint operator
$u\mapsto -h^2u''+Vu$ in $L^2(\Omega)$ with Dirichlet boundary conditions. By $1_\Omega$ we denote 
the indicator function of $\Omega$.

We need a couple of preliminary assertions.

\begin{lem}\label{lem:comptvpeps} Let $\Omega\subset\RR$ be a non-empty open interval 
and $V_1,V_2 \in L^\infty(\Omega)\cap L^1(\Omega)$ be real-valued, then for any $h>0$, $ E>0$ and $\varepsilon\in(0,1)$
one has
\begin{multline*}
\cN\big(h^2D^2+V_1+V_2,\Omega,-E\big)\\
\le \cN\Big(h^2D^2+\dfrac{1}{1-\varepsilon}\,V_1,\Omega,-E\big)
+
\cN\Big(h^2D^2 +\dfrac{1}{\varepsilon}\,V_2,\Omega,-E\Big)
\end{multline*}
and
\begin{multline*}
\cN\big(h^2D^2+V_1+V_2,\Omega,-E\big)\\
\ge 
\cN\Big(h^2D^2+ (1-\varepsilon)V_1, \Omega, -E\Big)
-
\cN\Big(h^2D^2-\dfrac{1-\varepsilon}{\varepsilon}\, V_2,\Omega,-E\Big).
\end{multline*}
\end{lem}
The proof is given in \cite[Proposition 5]{KS88} for $\Omega=\RR$ and extends literally 
to the case of an arbitrary interval $\Omega$.

\begin{lem}\label{lem:equivpertur}
Let  $a>0$, $C>0$ and $\gamma\in(0,2)$. For $h>0$ we set
\begin{equation}
        \label{eq-AB}
A\equiv A(h) = \dfrac{\sqrt{C}}{h^{1-\frac{\gamma}{2}}}\,a,
\quad
B\equiv B(h) = \dfrac{1}{2h},
\end{equation}
then
	\[
		\cN\Big(D^2 -\dfrac{1}{4h^2x^2}-\one_{(A,B)}, (A,+\infty),-1\Big) \sim
		\frac{\gamma}{4\pi}\frac{|\log h|}{h} \text{ for } h\to 0^+.
	\]
\end{lem}

\begin{proof}
Due to the min-max principle we have
\begin{multline*}
\cN(h):= \cN\Big(D^2 -\dfrac{1}{4h^2x^2}-\one_{(A,B)}, (A,+\infty),-1\Big)\\
=\cN\Big(D^2 -\dfrac{1}{4h^2x^2}-\one_{(A,B)}, (A,B),-1\Big)\\
+\cN\Big(D^2 -\dfrac{1}{4h^2x^2}-\one_{(A,B)}, (B,+\infty),-1\Big) +\sigma_1(h)
\end{multline*}
with $\sigma_1(h)\in\{0,1\}$, and for the  second term on the right-hand side we have
\begin{multline*}
\cN\Big(D^2 -\dfrac{1}{4h^2x^2}-\one_{(A,B)}, (B,+\infty),-1\Big)=\cN\Big(D^2 -\dfrac{1}{4h^2x^2}, (B,+\infty),-1\Big)\\
\le \cN\Big(D^2 -\dfrac{1}{4h^2B^2}, (B,+\infty),-1\Big)
=\cN\Big(D^2, (B,+\infty),0\Big)=0,
\end{multline*}
hence,
\begin{multline*}
\cN(h)=\cN\Big( D^2 -\dfrac{1}{4h^2x^2}-1, (A,B),-1\Big)+\sigma_1(h)\\
=\cN\Big(h^2 D^2 -\dfrac{1}{4h^2x^2}, (A,B),0\Big) +\sigma_1(h).
\end{multline*}
Remark that
	\[
		-u''(x) - \frac{1}{4h^2x^2} u(x) = 0
		\text{ for }
		u(x) = \sqrt{x}\sin\Big(\dfrac{\sqrt{1-h^2}}{2h}\log x\Big),
	\]
therefore, using Sturm's oscillation theorem, see e.g. \cite[Theorem 3.4]{Sim05} we obtain
	\begin{multline*}
		\cN\Big( D^2 -\dfrac{1}{4h^2 x^2}, (A,B),0\Big)
		=\#\big\{x\in (A,B): u(x) = 0\big\}\\
		= \#\Big\{ k \in\ZZ : \dfrac{\sqrt{1-h^2}}{2 h} \log A
		< \pi k <
		\dfrac{\sqrt{1-h^2}}{2\pi h} \log B\Big\}
		= \dfrac{\sqrt{1-h^2}}{2 \pi h} \log \dfrac{B}{A}+\sigma_2(h)
	\end{multline*}
with $\sigma_2(h)\in\{-1,0\}$. Therefore,
	\[
		\Bigg|\,\cN(h) - \dfrac{\sqrt{1-h^2}}{2\pi h} \log \dfrac{B}{A}\Bigg|\leq 2,
	\]
and an elementary computation yields
	\[
		\frac{\sqrt{1-h^2}}{2\pi h} \log \dfrac{B(h)}{A(h)}\sim  \frac{\gamma}{4\pi} \frac{|\log h|}{h} \text{ for } h\to 0^+,
	\]
which concludes the proof.
\end{proof}

\begin{lem}\label{lem:equivlem2} Let  $a>0$, $C>0$ and $\gamma\in(0,2)$,
then
	\[
		\cN \Big(h^2D^2 - \dfrac{1}{4x^2},(a,+\infty),-Ch^\gamma\Big) \sim \dfrac{\gamma}{4\pi}\dfrac{|\log h|}{h}
		\text{ for } h\to 0^+.
	\]
\end{lem}

\begin{proof} We continue using the notation \eqref{eq-AB}. Consider the unitary map
\[
		V: u \in L^2(A,+\infty) \to L^2(a,+\infty),
		\quad
		(Vu)(x)
		= \bigg(\dfrac{\sqrt{C}}{h^{1-\frac{\gamma}{2}}}\bigg)^{\frac{1}{2}} u\bigg(\dfrac{\sqrt{C}}{h^{1-\frac{\gamma}{2}}}x\bigg),
\]
then in view of
\[
V^{-1}\Big( -h^2\dfrac{d^2}{dx^2}-\dfrac{1}{4x^2}\Big)V=Ch^\gamma\Big(-\dfrac{d^2}{dx^2}-\dfrac{1}{4h^2x^2}\Big)
\]
one has
\[
\cN(h):=\cN \Big(h^2D^2 - \dfrac{1}{4x^2},(a,+\infty),-Ch^\gamma\Big)
=\cN \Big(D^2 - \dfrac{1}{4h^2x^2},(A,+\infty),-1\Big).
\]
Let us denote
\begin{equation*}
\begin{aligned}
\varepsilon&\equiv\varepsilon(h) := \dfrac{1}{|\log h|},\\
h_+&:=h\sqrt{1-\varepsilon}, & A_+&:=A(h_+), & B_+&:=B(h_+),\\
h_-&:=\dfrac{h}{\sqrt{1-\varepsilon}}, & A_-&:=A(h_-), & B_-&:=B(h_-).
\end{aligned}
\end{equation*}
Now we prove separately the upper and lower bounds for $\cN(h)$.

To obtain an upper bound for $\cN(h)$, we apply first Lemma~\ref{lem:comptvpeps},
\begin{multline*}
\cN(h)=
\cN \bigg(D^2 +\Big(- \dfrac{1}{4h^2x^2}-(1-\varepsilon)\one_{(A,B^+)}\Big)  + (1-\varepsilon)\one_{(A,B^+)},(A,+\infty),-1\bigg)\\
\le
\cN \Big(D^2-\dfrac{1}{4h_+^2x^2}-\one_{(A,B^+)},(A,+\infty),-1\Big)\\
+\cN \Big(D^2+ \dfrac{1-\varepsilon}{\varepsilon}\one_{(A,B^+)},(A,+\infty),-1\Big),
\end{multline*}
and one remarks that the last term is equal to $0$. Furthermore, one has $A_+>A$, and the min-max principle
gives
\begin{multline*}
\cN \Big(D^2-\dfrac{1}{4h_+^2x^2}-\one_{(A,B^+)},(A,+\infty),-1\Big)\\
=
\cN \Big(D^2-\dfrac{1}{4h_+^2x^2}-\one_{(A,B^+)},(A,A_+),-1\Big)\\
+
\cN \Big(D^2-\dfrac{1}{4h_+^2x^2}-\one_{(A,B^+)},(A_+,+\infty),-1\Big)
+\sigma_1(h)
\end{multline*}
with $\sigma_1(h)\in\{0,1\}$. For the first term on the right-hand side we have
\begin{multline*}
\cN \Big(D^2-\dfrac{1}{4h_+^2x^2}-\one_{(A,B^+)},(A,A_+),-1\Big)
=\cN \Big(D^2-\dfrac{1}{4h_+^2x^2},(A,A_+),0\Big)\\
\le \cN \Big(D^2,(A,A_+),\dfrac{1}{4h_+^2 A^2}\Big)
= \#\Big\{
n\in \NN: \dfrac{\pi^2 n^2}{(A_+-A)^2}<\dfrac{1}{4h_+^2 A^2}
\Big\}\\
\le \dfrac{1}{2\pi h_+} \dfrac{A_+-A}{A}=\dfrac{1}{2\pi h\sqrt{1-\varepsilon}}
\Big((1-\varepsilon)^{-\frac{1}{2}(1-\frac{\gamma}{2})}-1\Big)\\
=
\cO \Big( \dfrac{1}{h |\log h|}\Big)=o\Big( \dfrac{|\log h|}{h }\Big),
\end{multline*}
while due to Lemma~\ref{lem:equivpertur} one has
\begin{multline*}
\cN \Big(D^2-\dfrac{1}{4h_+^2x^2}-\one_{(A,B^+)},(A_+,+\infty),-1\Big)\\
=\cN \Big(D^2-\dfrac{1}{4h_+^2x^2}-\one_{(A^+,B^+)},(A_+,+\infty),-1\Big)\\
\sim \dfrac{\gamma}{4\pi} \dfrac{|\log h_+|}{h_+}\sim \dfrac{\gamma}{4\pi} \dfrac{|\log h|}{h}
\text{ for } h\to 0^+.
\end{multline*}
Therefore,
\[
\cN(h)\le \dfrac{\gamma}{4\pi} \dfrac{|\log h|}{h}\Big(1+o(1)\Big)
\text{ for } h\to 0^+.
\]

To obtain a lower bound for $\cN(h)$ we use again Lemma \ref{lem:comptvpeps},
\begin{multline}
    \label{eq-min1}
\cN(h)=\cN\bigg(
D^2+\Big(- \frac1{4h^2x^2} - \dfrac{1}{1-\varepsilon}\,\one_{(A,B_-)}\Big) + \dfrac{1}{1-\varepsilon}\,\one_{(A,B_-)},(A,+\infty),-1
\bigg)\\
\ge
\cN\Big(D^2- \frac1{4h_-^2x^2} - \one_{(A,B_-)},(A,+\infty),-1\Big)\\
-
\cN\Big(D^2- \dfrac{1}{\varepsilon}\, \one_{(A,B_-)},(A,+\infty),-1\Big).
\end{multline}
The last term can be easily estimated using the min-max principle
and the positivity of the Dirichlet Laplacian, which gives, with some $\sigma_2(h)\in\left\{0,1\right\}$,
\begin{multline*}
\cN\Big(D^2- \dfrac{1}{\varepsilon}\, \one_{(A,B_-)},(A,+\infty),-1\Big)\\
\le\cN\Big(D^2- \dfrac{1}{\varepsilon}\, \one_{(A,B_-)},(A,B_-),-1\Big)
+\cN\Big(D^2- \dfrac{1}{\varepsilon}\, \one_{(A,B_-)},(B_-,+\infty),-1\Big)
+\sigma_2(h)\\
=\cN\Big(D^2,(A,B_-),\dfrac{1-\varepsilon}{\varepsilon}\Big)
+ \cN\Big(D^2,(B_-,+\infty),-1\Big)
+\sigma_2(h)\\
=\cN\Big(D^2,(A,B_-),\dfrac{1-\varepsilon}{\varepsilon}\Big)+\sigma_2(h)
\le \dfrac{B_- -A}{\pi} \sqrt{\dfrac{1-\varepsilon}{\varepsilon}}
+\sigma_2(h)\\
=\cO\Big( \dfrac{\sqrt{|\log h|}}{h}\Big)=o\Big( \dfrac{|\log h|}{h}\Big).
\end{multline*}
It remains to estimate the first term on the right-hand side of \eqref{eq-min1}. As $A_-<A$,
using the min-max-principle one can estimate, with $\sigma_3(h)\in\{0,1\}$,
\begin{multline*}
\cN\Big(D^2- \frac1{4h_-^2x^2} - \one_{(A_-,B_-)},(A_-,+\infty),-1\Big)\\
=\cN\Big(D^2- \frac1{4h_-^2x^2} - \one_{(A_-,A)},(A_-,A),-1\Big)\\
+\cN\Big(D^2- \frac1{4h_-^2x^2} - \one_{(A,B_-)},(A,+\infty),-1\Big)+\sigma_3(h),
\end{multline*}
which implies
\begin{multline*}
\cN\Big(D^2- \frac1{4h_-^2x^2} - \one_{(A,B_-)},(A,+\infty),-1\Big)\\
\ge
\cN\Big(D^2- \frac1{4h_-^2x^2} - \one_{(A_-,B_-)},(A_-,+\infty),-1\Big)\\
-\cN\Big(D^2- \frac1{4h_-^2x^2} - \one_{(A_-,A)},(A_-,A),-1\Big)-1.
\end{multline*}
Due to Lemma~\ref{lem:equivpertur} one has
\[
\cN\Big(D^2- \frac1{4h_-^2x^2} - \one_{(A_-,B_-)},(A_-,+\infty),-1\Big)\sim \dfrac{\gamma}{4\pi} \dfrac{|\log h_-|}{h_-}
\sim \dfrac{\gamma}{4\pi} \dfrac{|\log h|}{h} \text{ for } h\to 0^+.
\]
Finally, by an explicit computation,
\begin{multline*}
\cN\Big(D^2- \frac1{4h_-^2x^2} - \one_{(A_-,A)},(A_-,A),-1\Big)\\
=\cN\Big(D^2- \frac1{4h_-^2x^2},(A_-,A),0\Big)\le \cN\Big(D^2- \frac1{4h_-^2A_-^2},(A_-,A),0\Big)\\
=\cN\Big(D^2,(A_-,A), \frac1{4h_-^2A_-^2}\Big)\le \dfrac{1}{2\pi h_-} \dfrac{A-A_-}{A_-}\\
=\dfrac{\sqrt{1-\varepsilon}}{2\pi h}\Big( (1-\varepsilon)^{-\frac{1}{2}(1-\frac{\gamma}{2})}-1\Big)
=\cO \Big( \dfrac{1}{h|\log h|}\Big)
=o \Big( \dfrac{|\log h|}{h}\Big),
\end{multline*}
which concludes the proof of the lower bound.
\end{proof}

\begin{proof}[\bf Proof of Proposition~\ref{prop:weyl}]
As changing the boundary condition at $a$ can alter the eigenvalue counting function at most by one, it is sufficient to consider
the case of Dirichlet boundary conditions. Let us pick $\varepsilon \in (0,1)$ and set
\[
h_+:=h\sqrt{1-\varepsilon}, \quad h_-:=\dfrac{h}{\sqrt{1-\varepsilon}}.
\]
To obtain the upper bound we remark that due to the assumption \eqref{eq-vdelta}
one can find $b>a$ such that
\[
-\dfrac{1}{4x^2}+v(x)\ge -\dfrac{1}{4(1-\varepsilon)x^2} \text{ for } x>b.
\]
Furthermore, denote
\[
A:=\Big\| -\dfrac{1}{4x^2}+v\Big\|_{L^\infty(a,b)},
\]
then due to the min-max principle one has
\begin{multline}
   \label{eqn:upbcount}
\cN \Big(h^2D^2-\dfrac{1}{4x^2}+v,(a,+\infty),-Ch^\gamma\Big)\\
\le
\cN \Big(h^2D^2-\dfrac{1}{4x^2}+v,(a,b),-Ch^\gamma\Big)
+
\cN \Big(h^2D^2-\dfrac{1}{4x^2}+v,(b,+\infty),-Ch^\gamma\Big)+1\\
\le 
\cN \Big(h^2D^2,(a,b),A\Big)
+\cN \Big(h^2D^2-\dfrac{1}{4(1-\varepsilon)x^2},(b,+\infty),-Ch^\gamma\Big)+1.
\end{multline}
An explicit computation gives
\[
\cN \Big(h^2D^2,(a,b),A\Big)\le \dfrac{(b-a)\sqrt{A}}{\pi h}= o\Big(\dfrac{|\log h|}{h}\Big),
\]
while due to Lemma \ref{lem:equivlem2} for $h\to 0^+$ one has
\begin{multline*}
\cN \Big(h^2D^2-\dfrac{1}{4(1-\varepsilon)x^2},(b,+\infty),-Ch^\gamma\Big)\\
=\cN \Big(h_+^2D^2-\dfrac{1}{4x^2},(b,+\infty),-C(1-\varepsilon)^{1-\frac{\gamma}{2}}h_+^\gamma\Big)\\
\sim \dfrac{\gamma}{4\pi} \dfrac{|\log h_+|}{h_+}\sim \dfrac{1}{\sqrt{1-\varepsilon}}\dfrac{\gamma}{4\pi} \dfrac{|\log h|}{h},
\end{multline*}
and the substitution into \eqref{eqn:upbcount} shows that
\begin{equation}
   \label{eq-up}
	\cN\Big(h^2D^2-\dfrac{1}{4x^2}+v,(a,+\infty),-Ch^\gamma\Big)
\le
\dfrac{\gamma}{4\pi} \dfrac{|\log h|}{h} \Big(\dfrac{1}{\sqrt{1-\varepsilon}}+o(1)\Big) \text{ as } h\to 0^+.
\end{equation}

For the lower bound we remark that one can find some $b>a$ such that
\[
-\dfrac{1}{4x^2}+v(x)\le -\dfrac{1-\varepsilon}{4x^2} \text{ for } x>b,
\]
then the min-max principle and Lemma~\ref{lem:equivlem2} yield
\begin{multline}\label{eq-down}
	\cN\Big(h^2D^2-\dfrac{1}{4x^2}+v,(a,+\infty),-Ch^\gamma\Big)
	\ge \cN\Big(h^2D^2-\dfrac{1}{4x^2}+v,(b,+\infty),-Ch^\gamma\Big)\\
	\ge
	\cN \Big(h_-^2D^2-\dfrac{1}{4x^2},(b,+\infty),-(1-\varepsilon)^{\frac{\gamma}{2}-1}Ch_-^{\gamma}\Big)\\
	\sim \dfrac{\gamma}{4\pi} \dfrac{|\log h_-|}{h_-}\sim \sqrt{1-\varepsilon}\,\dfrac{\gamma}{4\pi} \dfrac{|\log h|}{h}
	\text{ as } h\to 0^+.
		\end{multline}
As $\varepsilon\in(0,1)$ in both \eqref{eq-up} and \eqref{eq-down} can be taken arbitrarily small,
the claim follows.
\end{proof}

\section{Proofs of the main results}\label{sec:proofth}

\subsection{Reduction to a domain independent on $\theta$}\label{sub:redindtheta} Before passing to the proof of Theorem \ref{thm:mainth} we reformulate the problem
in a domain independent of $\theta$. To this aim, consider the unitary transform $V_\theta: L^2(\mathbb R^3) \to L^2(\mathbb R^3)$
given by
\[
(V_\theta u)(x_1,x_2,x_3):= 
\cos^2\theta \sqrt{\dfrac{2\sqrt 2}{\sin \theta}}\ u\left(\sqrt{2}\cos\theta x_1,\ \sqrt{2}\cos\theta x_2, \frac{\sqrt{2}\cos^2\theta}{\sin\theta}x_3\right),
\]
then $V_\theta$ is an isomorphism of $H^1(\mathbb R^3)$, and for $u\in H^1(\mathbb{R}^3)$ one has
\begin{equation}\label{eqn:unitequiv}
		\ell_\theta(u,u) = \frac{1+\tan^2\theta}{2} \,q_{\tan \theta}(V_\theta u,V_\theta u),
\end{equation}
where the sesquilinear form $q_h$ is defined on $H^1(\mathbb{R}^3)$ by
	\begin{equation}\label{eqn:defqtilde}
		q_h (u,u) = \iiint_{\mathbb{R}^3}\Big(h^2|\partial_{x_3}u|^2 + |\partial_{x_1} u|^2 + |\partial_{x_2} u|^2\Big)dx_1dx_2dx_3 - \iint_{C_{\frac{\pi}{4}}}|u|^2d\sigma.
	\end{equation}
Let $Q_h$ be the self-adjoint operator associated with the sesquilinear form $q_h$, then due to \eqref{eqn:unitequiv} 
one has the unitary equivalence
	\begin{equation*}
		L_\theta= \frac{1+\tan^2\theta}{2}\, V_\theta^{-1} Q_{\tan \theta} V_\theta,
	\end{equation*}
and
\[
	\sigma_{\rm ess}(Q_h) = \Big[-\frac{1}{2(1+h^2)},+\infty\Big).
\]
As mentioned in the introduction, the discrete spectrum of $Q_h$ is infinite for any $h$, in particular
\begin{equation}
   \label{enh}
E_n(Q_h)<-\dfrac{1}{4} \text{ for  all } n\in \NN \text{ and } h\in(0,1).
\end{equation}
Due to the elementary asymptotics
\[
\dfrac{1+\tan^2\theta}{2}= \dfrac{1}{2}+\mathcal O(\theta^2) \text{ as } \theta\to 0,
\]
Theorem \ref{thm:mainth} and Theorem \ref{thm:mainth2} become consequences of the following results
for $Q_h$:
\begin{prop}\label{propmain}
\textup{(a)} For each fixed $n\in \mathbb N$ and $h\to 0^+$ one has
	\[
		E_n(Q_h) = -2a_0 + 2(2n-1) a_1 h + \mathcal{O}(h^{\frac{3}{2}}).
	\]
\textup{(b)} For any $\gamma\in(0,\frac{3}{2})$ and any $C>0$ there holds
	\[
		\cN(Q_h,-Ch^\gamma)= \dfrac{\gamma}{4\pi }\dfrac{|\log h|}{h} + o\Big( \dfrac{\log h}{h}\Big) \text{ for } h\to 0^+.
	\]
\end{prop}

The analysis will be essentially based on a special decomposition of the form~$q_h$.
Namely, recall that the forms $b_{R,\beta}$, the operators $B_{R,\beta}$ and the eigenvalues $\mu_j(R;\beta)$ are defined
in Subsection~\ref{circ2}, and in the present section we denote
\[
w_R:=b_{R,\sqrt 2}, \quad W_R:=B_{R,\sqrt{2}},\quad \mu_j(R):=\mu_j(R,\sqrt 2).
\]
Remark that due to Proposition~\ref{prop:2Dmodel} the function $(0,+\infty)\ni R\mapsto \mu_1(R)$
satisfies
\begin{equation}
  \label{mu01}
	\lim_{R\to 0} \mu_1(R)=0, \quad
	\mu_1(R)= -\dfrac{1}{2}-\dfrac{1}{4R^2}+\cO\Big(\dfrac{1}{R^4}\Big),
\end{equation}
and has a unique minimum at $R=\xi_0$ with
\begin{equation}
    \label{muxi}
\mu_1(\xi_0)=-2a_0, \quad \mu_1''(\xi_0)=8a_1^2.
\end{equation}
For subsequent constructions, let us pick an arbitrary $\rho\in(0,\xi_0)$ such that
\begin{equation} 
  \label{muh}
\mu_1(R)>-\dfrac{1}{4} \text{ for } R\in (0,\rho),
\end{equation}
which exists due to the above properties of $\mu_1$.

Let $\Phi_R$ denote the normalized non-negative eigenfunction of $C_R$
for the eigenvalue $\mu_1(R)$. In virtue of  the spectral theorem and Proposition~\ref{prop:2Dmodel}(d) we have
\begin{equation}
    \label{spth}
w_R(v,v)\ge -\dfrac{1}{2} \|v\|^2_{L^2(\RR^2)} \text{ for any } v\in H^1(\RR^2)
\text{ with } \big\langle v,\Phi_R\big\rangle_{L^2(\RR^2)}=0.
\end{equation}

Furthermore, by Fubini's theorem one can can represent
	\begin{equation}\label{eqn:qtilredu}
		\begin{split}
		q_h(u,u) 	=& \iiint_{\mathbb{R}^2\times\mathbb{R}_-}\big(h^2 |\partial_{x_3}u|^2 + |\partial_{x_1} u|^2 + |\partial_{x_2} u|^2\big)dx_1dx_2dx_3\\
					& + h^2\iiint_{\mathbb{R}^2\times\mathbb{R}_+}|\partial_{x_3}u|^2dx_1dx_2dx_3
					 + \int_{0}^{\infty} w_{x_3}\big(u(\cdot,\cdot,x_3),u(\cdot,\cdot,x_3)\big) dx_3.
		\end{split}
	\end{equation}

\subsection{Upper bound for the eigenvalues of $Q_h$}

For $h>0$, denote by $T_h$
the operator in $L^2(\rho,+\infty)$ acting as
\[
(T_h f)(t)=-h^2f''(t)+\mu_1(t) f(t)
\]
on the  functions satisfying the Dirichlet boundary condition at $\rho$.
Remark that the associated sesquilinear form is
\[
t_h(f,f)=\int_\rho ^\infty \Big(h^2|f'|^2 +\mu_1(t) |f|^2\Big)dt, \quad f\in H^1_0(\rho,\infty).
\]
The goal of this subsection is to obtain the following upper bound for the eigenvalues of $Q_h$
in terms of $T_h$. Recall that due to the result of \cite{KS88} and the asymptotics \eqref{mu01}
the operator $T_h$ has infinitely many eigenvalues in $(-\infty,-\frac{1}{2})$ for any $h\in(0,1)$. 
\begin{prop}\label{uprop}
There exists $M>0$ such that for any $n\in\NN$ and any $h\in(0,1)$
there holds $E_n(Q_h)\le E_n(T_h)+Mh^2$.
\end{prop}

\begin{proof}
Let $f\in H^1_0(\rho,+\infty)$, then the function $u_f$,
\[
u_f(x_1,x_2,x_3)=f(x_3)\Phi_{x_3}(x_1,x_2),
\]
belongs to $H^1(\RR^3)$ due to Proposition~\ref{prop:2Dmodel}(a,c), and one has
\begin{equation}
   \label{ufug}
\langle u_f,u_g\rangle_{L^2(\RR^3)}=\langle f,g\rangle_{L^2(\rho,\infty)} \text{ for } f,g\in H^1_0(\rho,\infty).
\end{equation}
Using the representation \eqref{eqn:qtilredu} we arrive at
	\begin{align*}
		q_h(u_f,u_f) &= h^2\iiint_{\mathbb{R}^2\times(a,b)} \big|f'(x_3)\Phi_{x_3}(x_1,x_2) + f(x_3)\partial_{x_3}\Phi_{x_3}(x_1,x_2)\big|^2 dx_1dx_2dx_3\\
		&\quad +\int_{\rho}^\infty\mu_1(x_3)|f(x_3)|^2dx_3.
	\end{align*}
Due to the normalization one has
	\[
		1 = \int_{\mathbb{R}^2}\Phi_{x_3}(x_1,x_2)^2dx_1dx_2, \quad x_3\in(\rho,\infty),
	\]
thus taking the derivative with respect to $x_3$ one obtains
	\[
		0 = \int_{\mathbb{R}^2} \Phi_{x_3}(x_1,x_2)\partial_{x_3}\Phi_{x_3}(x_1,x_2) dx_1dx_2,
	\]
which yields
\[
2\Re \iiint_{\mathbb{R}^2\times(\rho,\infty)}
\overline{f'(x_3)} f(x_3) \Phi_{x_3}(x_1,x_2)\partial_{x_3}\Phi_{x_3}(x_1,x_2)\, dx_1dx_2dx_3=0.
\]
and, consequently,
	\begin{align*}
		q_h(u_f,u_f)= &h^2\iiint_{\mathbb{R}^2\times(\rho,+\infty)}\big|f'(x_3)\big|^2\Phi_{x_3}(x_1,x_2)^2 dx_1dx_2dx_3\\
		& + h^2\iiint_{\mathbb{R}^2\times(\rho,+\infty)}\big|f(x_3)\big|^2 \big|\partial_{x_3}\Phi_{x_3}(x_1,x_2)|^2 dx_1dx_2dx_3\\
		& + \iiint_{\mathbb{R}^2\times(\rho,+\infty)}\mu_1(x_3) \big|f(x_3)\big|^2 dx_1dx_2dx_3\\
		=& \int_\rho^\infty \Big(h^2\big|f'(x_3)\big|^2 + h^2 m(x_3)\big|f(x_3)\big|^2 +\mu_1(x_3) \big|f(x_3)\big|^2\Big)dx_3
	\end{align*}
where we set $m(x_3) :=\|\partial_{x_3} \Phi_{x_3}\|^2_{L^2(\RR^2)}$. By virtue of Proposition~\ref{prop:2Dmodel}(c) one has $M:=\sup_{x_3>\rho} m(x_3)<\infty$, hence,
\[
q_h(u_f,u_f)\le t_h(f,f)+ Mh^2 \|f\|^2_{L^2(\rho,\infty)}.
\]
Now using the min-max principle and \eqref{ufug} we have, for any $n\in\NN$,
\begin{multline*}
E_n(Q_h)=\inf_{\substack{V\subset \cD(q_h)\\ \dim V=n}}\sup_{\substack{u\in V\\u\ne 0}} \dfrac{q_h(u,u)}{\|u\|^2_{L^2(\RR^3)}}
\le \inf_{\substack{V\subset \cD(t_h)\\ \dim V=n}}\sup_{\substack{f\in V\\f\ne 0}} \dfrac{q_h(u_f,u_f)}{\|u_f\|^2_{L^2(\RR^3)}}\\
\le \inf_{\substack{V\subset \cD(t_h)\\ \dim V=n}}\sup_{\substack{f\in V\\f\ne 0}} \dfrac{t_h(f,f)+Mh^2 \|f\|^2_{L^2(\rho,\infty)}}{\|f\|^2_{L^2(\rho,\infty)}}
=E_n(T_h)+Mh^2. \qedhere
\end{multline*}
\end{proof}

\subsection{Lower bound for the eigenvalues of $Q_h$}

Denote by $s_h$ the sesquilinear form in $L^2(\rho,\infty)$ given by
\[
s_h(f,f)=\int_\rho^\infty \big(h^2|f'|^2 + \mu_1 |f|^2\big)dt, \quad f\in H^1(\rho,\infty),
\]
and by $S_h$ the associated self-adjoint operator in $L^2(\rho,\infty)$, which
acts as
\[
S_h f= -h^2f''+\mu_1 f
\]
with Neumann boundary condition at $\rho$. Recall that due to the result of \cite{KS88} and the asymptotics \eqref{mu01}
the operator $S_h$ has infinitely many eigenvalues in $(-\infty,-\frac{1}{2})$ for any $h\in(0,1)$. 

\begin{prop}\label{downprop}
Denote $\hbar:=h(1+h^{\frac{1}{2}})^{\frac{1}{2}}$, then there exists $M>0$
such that
\[
E_n(Q_h)\ge E_n(S_\hbar)-Mh^{\frac{3}{2}}
\]
for $\hbar\in(0,1)$ and $n\in\NN$.
\end{prop}

\begin{proof}
Denote $\Omega_0 = \mathbb{R}^2\times(\rho,\infty)$, $\Omega_1 = \mathbb{R}^2\times(-\infty,0)$,
$\Omega_2 = \mathbb{R}^2\times(0,\rho)$.
For $u\in H^1(\mathbb{R}^3)$ and $j\in\{0,1,2\}$ we set $u_j = u|_{\Omega_j}$, then
$q_h(u,u) = \sum_{j\in\{0,1,2\}} q_{j,h}(u_j,u_j)$ with sesquilinear forms $q_{j,h}$ defined on $H^1(\Omega_j)$ by
	\[
		q_{j,h}(u,u) =  \iiint_{\Omega_j}\big(h^2|\partial_{x_3}u|^2 + |\partial_{x_1} u|^2 + |\partial_{x_2} u|^2\big)dx_1dx_2dx_3 - \iint_{\Omega_j\cap C_{\frac{\pi}{4}}} |u|^2d\sigma.
	\]
Furthermore, define the decoupled sesquilinear form
	\[
		\widetilde q_h(u,u) = \sum_{j\in\{0,1,2\}} q_{j,h}(u_j,u_j),\quad u=(u_j)_{j\in\{0,1,2\}}\in \oplus_{j\in\{0,1,2\}} H^1(\Omega_j).
	\]
By construction one has $q_h \geq \widetilde q_h$, and by the min-max principle one has
$E_n(q_h) \geq E_n(\widetilde q_h)$. Remark that $E_n(\widetilde q_h)$
is the $n$-th smallest element of the disjoint union $\sqcup_{(n,j)\in \NN\times\{0,1,2\}} \big\{E_n(q_{j,h})\big\}$
and that due to \eqref{enh} there holds
\[
E_n(\Tilde q_h)<-\dfrac{1}{4} \text{ for all } n\in\NN \text{ and } h\in(0,1).
\]
Notice that for $u \in H^1(\Omega_1)$ one has $q_{1,h}(u,u) \geq 0$, while for $u\in H^1(\Omega_2)$ there holds
\begin{align*}
q_{2,h}(u,u)&=h^2\iiint_{\mathbb{R}^2\times(0,\rho)}|\partial_{x_3}u|^2dx_1dx_2dx_3
					 + \int_{0}^{\rho} w_{x_3}\big(u(\cdot,\cdot,x_3),u(\cdot,\cdot,x_3)\big) dx_3\\
					&\ge  h^2\iiint_{\mathbb{R}^2\times(0,\rho)}\Big(|\partial_{x_3}u|^2 + \mu_1(x_3) \big|u\big|^2\Big)dx_1dx_2dx_3\\
					&\ge \inf_{x_3\in (0,\rho)} \mu_1(x_3) \|u\|^2_{L^2(\Omega_2)}\ge -\dfrac{1}{4}\,\|u\|^2_{L^2(\Omega_2)},
\end{align*}
where the last inequality holds due to the condition \eqref{muh} for the choice of $\rho$. 
It follows that $E_n(\widetilde q_h)=E_n(q_{0,h})$ and then
\begin{equation*}
E_n(Q_h)\ge E_n(q_{0,h}) \text{ for all } n\in\NN, \quad h\in (0,1).
\end{equation*}
Now let us now focus on the quadratic form $q_{0,h}$. Introduce an orthogonal projector
$P$ in $L^2(\Omega_0)$ by
	\[
		P u(x_1,x_2,x_3) = f(x_3)\Phi_{x_3}(x_1,x_2),\quad \text{where } f(x_3) = \langle u(\cdot,\cdot,x_3),\Phi_{x_3}\rangle_{L^2(\mathbb{R}^2)}
	\]
and set $P^\perp:=1-P$. Remark that $\|Pu\|_{L^2(\Omega_0)}=\|f\|_{L^2(\rho,\infty)}$.

For $u\in H^1(\Omega_0)$ using \eqref{spth} one obtains
	\begin{equation}\label{eqn:lbqh}
	\begin{split}
		q_{0,h}(u,u) 	=&\ h^2\|\partial_{x_3} u\|_{L^2(\Omega_0)}^2 + \int_{\rho}^\infty w_{x_3}(u,u) dx_3\\
					=&\ h^2\Big(\|P\partial_{x_3} u\|_{L^2(\Omega_0}^2 + \|P^\perp\partial_{x_3} u\|_{L^2(\Omega_0}^2\Big)\\
					&\quad + \int_{\rho}^\infty \Big(w_{x_3}(Pu,Pu) + w_{x_3}(P^\perp u,P^\perp u)\Big)dx_3\\
					&\geq h^2\|P\partial_{x_3} u\|_{L^2(\Omega_0)}^2 + \int_\rho^\infty \mu_1(x_3)\big|f(x_3)\big|^2dx_3
					- \frac12 \|P^\perp u\|_{L^2(\Omega_0}^2
	\end{split}
	\end{equation}
For a.e. $x_3\in (\rho,\infty)$ we have $P\partial_{x_3} u = f'(x_3) \Phi_{x_3}
- \langle u,\partial_{x_3}\Phi_{x_3}\rangle_{L^2(\RR^2)}\Phi_{x_3}$.
Consequently, for $\varepsilon \in (0,1)$ we obtain
	\begin{equation}\label{eqn:lbcomutproj}
	\begin{split}
		\|P\partial_{x_3} u\|_{\Omega_0}^2	\geq&\ (1+\varepsilon)\|f'\|_{L^2(\rho,\infty)}^2 +(1-\varepsilon^{-1})\int_{\rho}^\infty\big|\langle u,\partial_{x_3}\Phi_{x_3}\rangle_{L^2(\RR^2)}\big|^2dx_3\\
							\geq&\ (1+\varepsilon)\|f'\|_{L^2(\rho,\infty)}^2 -\varepsilon^{-1}\int_{\rho}^\infty\big|\langle u,\partial_{x_3}\Phi_{x_3}\rangle_{L^2(\RR^2)}\big|^2dx_3\\
							\geq&\ (1+\varepsilon)\|f'\|_{(L^2(\rho,\infty)}^2
							-\varepsilon^{-1}\int_{\rho}^\infty\big\|u(\cdot,x_3)\big\|^2_{L^2(\RR^2)}
							\cdot \big\|\partial_{x_3}\Phi_{x_3}\big\|_{L^2(\RR^2)}^2dx_3\\
							\geq &\ (1+\varepsilon)\|f'\|_{L^2(\rho,\infty)}^2-M\varepsilon^{-1}\big\|u\big\|^2_{L^2(\Omega_0)}
	\end{split}
	\end{equation}
where  $M = \sup_{x_3\in(\rho,\infty)}\|\partial_{x_3}\Phi_{x_3}\|_{\mathbb{R}^2}^2$ is finite
thanks to Proposition \ref{prop:2Dmodel}(c).
Combining \eqref{eqn:lbqh} and \eqref{eqn:lbcomutproj} and choosing $\varepsilon = h^{\frac{1}{2}}$, we obtain 
	\begin{multline*}
		q_{0,h}(u,u)+M h^{\frac{3}{2}}\|u\|^2_{L^2(\Omega_0)}\\
		\geq
			 h^2(1+h^{\frac{1}{2}})\|f'\|_{L^2(\rho,\infty)}^2
		+ \int_\rho^\infty \mu_1(x_3)|f(x_3)|^2dx_3
		- \frac12 \|P^\perp u\|_{L^2(\Omega_0)}^2,
	\end{multline*}
which rewrites as
	\begin{equation*}
		q_{0,h}(u,u) +M h^{\frac{3}{2}}\|u\|^2_{L^2(\Omega_0)}	\geq s_{\hbar}(f,f) - \dfrac{1}{2}\|P^\perp u\|_{L^2(\Omega_0)}^2.
	\end{equation*}
Introduce a new quadratic form $a_{h}$ in $L^2(\rho,\infty)\oplus \ran P^\perp$
defined for $(f,v)\in H^1(\rho,\infty)\oplus \ran P^\perp$ by
	\[
		a_h\big((f,v),(f,v)\big) = s_\hbar(f,f) - \dfrac{1}{2}\|v\|_{L^2(\Omega_0)}^2.
	\]
The map $V:L^2(\Omega_0)\to L^2(\rho,\infty)\oplus \ran P^\perp$, $u\mapsto (f,P^\perp u)$,
is unitary, and as just shown we have
\[
q_{0,h}(u,u)+M h^{\frac{3}{2}}\|u\|^2_{L^2(\Omega_0)}\ge a_{h}(Vu,Vu).
\]
This implies by the min-max principle
$E_n(q_{0,h})\ge  E_n(A_h) -Mh^{\frac{3}{2}}$ for any $n\in \NN$,
where $A_h$ is the operator associated with $a_{h}$, which is simply
$A_h=S_\hbar\oplus \big(-\frac{1}{2}\big)$.
As noted above, for $\hbar\in (0,1)$ the operator $S_\hbar$ has infinitely many eigenvalues
in $(-\infty,-\frac{1}{2})$, therefore, $E_n(A_h)=E_n(S_\hbar)$
for any $n\in\NN$.
\end{proof}

\subsection{Proof of Proposition~\ref{propmain}}

(a) Remark that $\hbar=h+\cO(h^{\frac{3}{2}})$ for small $h$.
By Proposition~\ref{prop:harmapprox}, for each fixed $n\in\NN$ and $h\to 0^+$
one has
\begin{align*}
E_n(T_h)&= \mu_1(\xi_0) + (2n-1)\sqrt{\frac{\mu_1''(\xi_0)}{2}}h + \cO(h^{\frac{3}{2}}),\\
E_n(S_\hbar)&= \mu_1(\xi_0) + (2n-1)\sqrt{\frac{\mu_1''(\xi_0)}{2}}\hbar + \cO(\hbar^{\frac{3}{2}})\\
&=\mu_1(\xi_0) + (2n-1)\sqrt{\frac{\mu_1''(\xi_0)}{2}}h + \cO(h^{\frac{3}{2}}).
\end{align*}
Substituting the values \eqref{muxi} and using Propositions~\ref{uprop} and \ref{downprop}
gives the result.

(b) Let $C>0$ and $\gamma\in(0,\frac{3}{2})$.
Using Proposition~\ref{uprop} one has, as $h\to 0^+$,
\begin{align*}
\cN\Big(Q_h,-\dfrac{1}{2}-Ch^\gamma\Big)&=\#\Big\{n\in \NN: E_n(Q_h)<-\dfrac{1}{2}-Ch^\gamma\Big\}\\
&\ge \#\Big\{n\in \NN: E_n(T_h)+Mh^{\frac{3}{2}}<-\dfrac{1}{2}-Ch^\gamma\Big\}\\
&=\#\Big\{n\in \NN: E_n\Big(T_h+\dfrac{1}{2}\Big)<-Ch^\gamma-Mh^{\frac{3}{2}}\Big\}\\
&\ge\#\Big\{n\in \NN: E_n\Big(T_h+\dfrac{1}{2}\Big)<-C'h^\gamma\Big\}\equiv \cN\Big(T_h+\dfrac{1}{2},-C'h^\gamma\Big),
\end{align*}
with an arbitrary $C'>C$. Due to \eqref{mu01}, the operator $T_h+\frac{1}{2}$
satisfies the assumptions of Proposition~\ref{prop:weyl}, hence,
\[
\cN\Big(T_h+\dfrac{1}{2},-C'h^\gamma\Big)\sim \dfrac{\gamma}{4\pi} \dfrac{|\log h|}{h},
\]
and
\[
\cN\Big(Q_h,-\dfrac{1}{2}-Ch^\gamma\Big)\ge \dfrac{\gamma}{4\pi} \dfrac{|\log h|}{h} \Big(1+o(1)\Big) \text{ for } h\to 0^+.
\]
Furthermore, with the help of Propositions~\ref{downprop} one estimates, for small $h$,
\begin{align*}
\cN\Big(Q_h,-\dfrac{1}{2}-Ch^\gamma\Big)&=\#\Big\{n\in \NN: E_n(Q_h)<-\dfrac{1}{2}-Ch^\gamma\Big\}\\
&\le \#\Big\{n\in \NN: E_n(S_\hbar)-Mh^{\frac{3}{2}}<-\dfrac{1}{2}-Ch^\gamma\Big\}\\
& \le \#\Big\{n\in \NN: E_n(S_\hbar)<-\dfrac{1}{2}-C''\hbar^\gamma\Big\} =\cN\Big(S_\hbar+\frac{1}{2},-C''\hbar^\gamma\Big)
\end{align*}
with an arbitrary $C''\in(0,C)$. Using Proposition~\ref{prop:weyl} applied
to $S_\hbar+\frac{1}{2}$ as well as $\hbar = h +\mathcal{O}(h^{\frac{3}{2}})$ for small $h$, we arrive at
\[
\cN\Big(Q_h,-\dfrac{1}{2}-c\hbar^\gamma\Big)\le \dfrac{\gamma}{4\pi} \dfrac{|\log h|}{h} \Big(1+o(1)\Big) \text{ for } h\to 0^+,
\]
which completes the proof.

\section*{Acknowledgments}

Thomas Ourmi\`eres-Bonafos is supported by a public grant as part of the ``Investissement d'avenir'' project, reference ANR-11-LABX-0056-LMH, LabEx LMH.
Thomas Ourmi\`eres-Bonafos and Konstantin Pankrashkin are also supported by the PHC Amadeus 2017--2018 37853TB.
Fabio Pizzichillo is supported by ERCEA, Advanced Grant project 669689 HADE, MINECO project MTM2014-53145-P and also by the Basque Government through the BERC 2014-2017 program and by the Spanish Ministry of Economy and Competitiveness MINECO: BCAM Severo Ochoa accreditation SEV-2013-0323.

The authors thank Serge Richard for advice on references on Bessel functions.

\end{document}